\documentclass{amsart}

\usepackage{url}

\usepackage{amscd}
\usepackage{amsfonts}
\usepackage{amssymb}
\usepackage{euscript}
\usepackage{amsmath}
\usepackage{amsthm}
\usepackage[matrix, arrow, curve]{xy}
\usepackage{mathrsfs}

\usepackage{dsfont}

\newcommand{\Li}{L_\infty}
\newcommand{\LL}{\mathop{L}}
\newcommand{\Alt}{\mathop{\mathrm{Alt}}}
\newcommand{\T}{T}
\newcommand{\TT}{{\mathrm T}}
\newcommand{\F}{F}
\newcommand{\R}{{{\mathbb R}^n}}
\newcommand{\Det}{\mathop{\mathrm{Det}}}
\newcommand{\sgn}{\mathrm{sgn}}
\newcommand{\C}{\mathscr{C}}
\newcommand{\g}{\mathfrak g}
\newcommand{\cc}{c}
\newcommand{\dtot}{\mathrm{d_{tot}}}
\newcommand{\Dn}{\mathop{\mathfrak D}\nolimits^n}
\newcommand{\Ue}{\mathop{U}\nolimits}
\newcommand{\ii}{\imath}
\newcommand{\Ab}{\mathrm{Ab}}
\newcommand{\set}[1]{\mathbf{[#1]}}

\theoremstyle{plain}
\newtheorem{prop}{Proposition}
\newtheorem{stat}{Statement}

\theoremstyle{definition}
\newtheorem{definition}{Definition}

\theoremstyle{remark}

\begin{document}

\title{Manifoldic homology and Chern-Simons formalism}
\author{Nikita Markarian}

\address{National Research University Higher School of Economics,
Department of Mathematics, 20 Myasnitskaya str., 101000, Moscow,
Russia}

\email{nikita.markarian@gmail.com}

\thanks{This work has been partially supported 
AG Laboratory GU-HSE, RF government grant, ag.~11.G34.31.0023 and the National Research University
Higher School of Economics' Academic Fund Program in 2012-2013,
research grant No.~11-01-0145, RFBR--12-01-00944.
}

\date{}

\maketitle

\section*{Introduction}

Manifoldic homology (we suggest this term instead of ``topological
chiral homology with constant coefficients'' from \cite{Lu}) is a
far-reaching generalization of  Hochschild homology. In the theory of
Hochschild and cyclic homology the additive Dennis trace map (e.~g. \cite[8.4.16]{Lo}) plays an
important role. Let $A$ be an associative algebra. Denote by
$\LL(A)$ the underlying Lie algebra of $A$. Then the additive  Dennis trace gives a  map from
$H_*(\LL(A))$ to $HH_{*}(A)$. The aim
of the present note is to generalize this morphism  to  any $e_n$-algebra
$A$.

Let $e_n$ be the operad of rational chains of the operad of little discs and $A$
be an algebra over it. The complex $A[n-1]$ is equipped with homotopy
Lie algebra structure, denote it by $\LL(A)$. Fix a compact oriented
$n$-manifold $M$  without boundary. For simplicity we restrict
ourselves to parallelizable manifolds, this restriction may be
removed by introducing framed little discs as in \cite{PS}. Denote
by $HM_*(A)$ the manifoldic homology of $A$ on $M$ introduced in Definition \ref{1} below,
and by $H_*(\LL(A))$ the Lie algebra  homology. In  the central
Proposition \ref{main} we give a morphism $H_*(\LL(A))\to
HM_*(A)$ explicitly in terms of the Fulton--MacPherson operad.

For $n=1$ and $M=S^1$, that is for homotopy associative algebras and  Hochschild homology, 
the above morphism may be presented
as the composition of natural morphisms $H_*(\LL(A))\to HH_*(\Ue(\LL(A))\to
HH_*(A)$, where $\Ue(-)$ is the universal enveloping algebra.
For $n>1$ the analogous statement holds, with the universal enveloping algebra replaced
by the universal enveloping $e_n$-algebra. The definition of the latter notion naturally appears in the context of
Koszul duality for $e_n$-algebras, which is still under construction, see nevertheless e.~g. \cite{Fr}
and references therein.
We need even more, than Koszul duality. The description of application of our construction to 
manifold invariants 
requires the Koszul duality for $e_n$-algebras with curvature.
These subjects are briefly discussed in the last section.

Our main construction is exemplary and may be generalized in many ways. 
For example,
one may take some modules over $A$ and put them into some points of $M$. 
Then one get a map from 
homology of $\LL(A)$
with coefficients in an appropriate module to the manifoldic homology with coefficients in these 
modules. 
In particular, if one take copies of $A$ itself as modules, then manifoldic homology with 
coefficient in them equals to the usual manifoldic homology; thus one get a map from homology of 
$\LL(A)$ 
with coefficients in the tensor product of adjoint modules to $HM_*(A)$. 
One needs this generalization to 
build a working theory of invariants of 3-manifolds,
I hope to treat this subject elsewhere.

\smallskip

{\bf Acknowledgments.} I am grateful to D.~Calaque, A.~Cattaneo,
G.~Ginot, A.~Khoroshkin and L.~Positselski for helpful discussions. My special
thanks to M.~Kapranov for  the inspiring discussion and  the term
``manifoldic homology''.

The present note is partially initiated by the work of K.~ Costello and O.~Gwilliam on factorization algebras in perturbative quantum field theory   (\cite{CG}), although it is hard to point at exact relations.  

\section{Trees and $\Li$}
\subsection{Trees}

A tree is an oriented connected graph with three type of vertices:
{\em root} has one incoming edge and no outgoing ones, {\em leaves}
have one outgoing edge and no incoming  ones and {\em internal
vertexes} have one outgoing edge and more than one incoming ones.
Edges incident to leaves will be called {\em inputs}, the edge incident
to the root will be called the {\em output} and all other edges will be called {\em internal
edges}. The degenerate tree has one edge and no internal vertexes.
Denote by $\T_k(S)$ the set of non-degenerate trees with $k$ internal
edges and leaves labeled by a set $S$.

For two trees $t_1\in T_{k_1}(S_1)$ and $t_2\in T_{k_2}(S_2)$ and an element $s\in S_1$
the composition of trees $t_1\circ_s t_2\in T_{k_1+k_2+1}$ is obtained by identification of the input of $t_1$
corresponding to $s$ and the output of $t_2$. Composition of trees is associative and
the degenerate tree is the unit.  The set of trees with respect to the composition forms an operad.

Call the tree with only one internal vertex the {\em star}.
Any  non-degenerate tree with $k$ internal edges may be uniquely presented as a composition of
$k+1$ stars.

The operation of {\em edge splitting} is the following: take
a non-degenerate tree, present it as a composition of stars and
replace one star  with a tree that is  a product of two stars and
has the same set of inputs . The operation of  an edge splitting
depends on a internal vertex and a proper subset of incoming edges.

\subsection{$\Li$}

For a non-degenerate tree $t$ denote by $\Det(t)$ the
one-dimensional $\mathbb Q$-vector space that is the determinant of
the vector space generated by internal edges. For $s>1$ consider the
complex
\begin{equation}
L(s)\colon \bigoplus_{t\in T_0(\set{s})} \Det (t) \to
\bigoplus_{t\in T_1(\set{s})} \Det(t) \to \bigoplus_{t\in
T_2(\set{s})} \Det (t) \to \cdots,
 \label{complex}
\end{equation}
where $\set{s}$ is the set of $s$ elements,
the cohomological degree of a tree $t\in T_k(\set{s})$ is $2-s+k$ and the  
differential is given by all possible splitting  of an edge (see e.~g.
\cite{GK}). The composition of trees equips the sequence
$L(i)\otimes \sgn$ with the structure of a $dg$-operad, here $\sgn$ is the
sign representation of the symmetric group.

This operad is called {\em $\Li$ operad}. Denote by $\Li[n]$ the
$dg$-operad given by the complex $L(s)[n(s-1)]\otimes(\sgn)^n$
and refer to it as $n$-shifted $L_\infty$ operad.

\section{Fulton--MacPherson operad}
\subsection{Fulton--MacPherson compactification}

The  Fulton--MacPherson compactification is introduced in \cite{GJ}
and \cite{Ma}, see also \cite{PS} and \cite{AS}. We cite here its
properties that are essential for our purposes.

For a finite set $S$ denote by $(\R)^S$ the set of ordered $S$-tuples in $\R$.
For a finite set $S$ denote by $\Delta_S\colon \R\to (\R)^S$ the diagonal embedding.
We will denote by $\set{n}$  the set of $n$ elements.

Let $\C^0_S(\R)\subset (\R)^S$ be the space of ordered pairwise distinct
points in $\R$ labeled by $S$. The {\em Fulton--MacPherson
compactification} $\C_S(\R)$ is a manifold with corners with interior
$\C_S^0(\R)$. The projection $\C_S(\R)\stackrel {\pi}{\to} (\R)^S$ is defined,
which is an isomorphism on  $\C^0_S(\R)\subset \C_S(\R)$. Moreover there is a
sequence of manifolds with corners $\F_n(S)$ labeled by finite sets and
maps $\phi_{S_1,\dots, S_k}$ that fit in the diagram
$$
\begin{CD}
\F_n(S_1)\times\cdots\times \F_n(S_k)\times \C_{\set{k}} (\R)@>\phi_{S_1,\dots, S_k}>> \C_{(S_1\cup\dots\cup S_k)}(\R)\\
@VV\pi V        @VV\pi V \\
(\R)^k @>{\Delta_{S_1}\times\cdots\times \Delta_{S_k}}>> (\R)^{(S_1\cup\dots\cup S_k)}
\end{CD}
$$
where  the left arrow is the projection to the point on the first factors
and $\pi$ on the last one. Restrictions of $\phi_{S_1,\dots, S_k}$
to $\F_n(S_1)\times\cdots\times \F_n(S_k)\times \C^0_{\set{k}}
(\R)$ are isomorphisms onto the image. It follows that
$\F_n(S)=\pi^{-1} \vec{0}$, where $\vec{0}\in(\R)^S$ is $S$-tuple
sitting at the origin. Being restricted on $\F_n(S)\subset \C_S(\R)$,
maps $\phi$ equip $\F_n(S)$ with an operad structure:
$$
\phi_{\set{s_1},\dots, \set{s_k}}\colon
\F_n(\set{s_1})\times\cdots\times \F_n(\set{s_k})\times \F_n(\set{k})
\to \F_n(\set{s_1+\dots+ s_k}).
$$
Manifolds $\C_{\set{k}}(\R)$ and $\F_n(\set{k})$ are equipped  with a $k$-th symmetric group action
consistent with its natural action on $\C_{\set{k}}^0(\R)$ and all maps  are compatible with this action.

\begin{definition}[\cite{PS}, \cite{GJ}, \cite{Ma}]  The sequence of spaces  $\F_n(\set{k})$
with the symmetric group action and composition morphisms as above is called the {\em
Fulton--MacPherson operad.}
\end{definition}

\subsection{Strata, trees  and $\Li$}

There is a map of sets $\F_n(S)\stackrel{\mu}{\to} T(S) $ that
subdivides $\F_n(S)$ into smooth strata. 
This map is totally defined by the following properties. Firstly, $\mu$ is consistent with the operad structure in
the sense that the  preimage of a composition is the composition of
preimages. Secondly, the zero codimension stratum corresponding to a
star tree is the  intersection of $\pi^{-1} \vec{0}$ and the stratum of 
$\C_{S}(\R)$ that is the blow-up of
the small diagonal minus pull backs of other diagonals.
These latter strata freely generate the Fulton--MacPherson operad as a set.

Denote by $C_*(\F_n)$ the $dg$-operad of rational chains of the Fulton--MacPherson operad.
For a tree $t\in\T(S)$ let $[\mu^{-1} (t)]\in C_*(\F_n(S))$ be the chain presented by its preimage under $\mu$.

\begin{prop}
Map $[\mu^{-1} (\cdot)]$ gives a morphism 
from shifted $L_\infty$ operad $L(s)[s(1-n)]$  
to the $dg$-operad $C_*(\F_n(\set{s}))$ of rational chains of the Fulton--MacPherson operad.
\label{tree}
\end{prop}

\begin{proof}
To see that the map commutes with the differential note, that two strata given by $\mu$ with 
dimensions differing by 1
are incident if and only if  one of the corresponding trees is obtained from another by  edge splitting.
In this way we get a basis in the conormal bundle to a stratum labeled by the internal edges.
It follows the consistency of the map from the statement with signs. 
\end{proof}
It follows that there is a morphism of $dg$-operads
\begin{equation}
\Li[1-n]\to C_*(\F_n)
\label{morphism}
\end{equation}

Let $e_n$ be the $dg$-operad of rational chains of the operad of little $n$-discs.

\begin{prop}
Operad $C_*(\F_n)$ is weakly homotopy equivalent to $e_n$.
\label{eq}
\end{prop}

\begin{proof}
See \cite[Proposition 3.9]{PS}.
\end{proof}

Thus there is a homotopy morphism of operads $\Li[1-n]\to e_n$.

\section{Manifoldic and Lie algebra homology}

\subsection{Manifoldic homology}

Let $M$ be an $n$-dimensional parallelized compact manifold without boundary.
 In the same way as for $\R$ there is the
Fulton-MacPherson compactification
$\C_S(M)$ of the space $\C_S^0(M)$ of ordered pairwise distinct points in $M$
labeled by S; inclusion $\C_S^0(M)\hookrightarrow \C_S(M)$ is a homotopy equivalence.
There is  a projection $\C_S(M) \stackrel{\pi}{\to} M^S$ and maps $\phi_{S_1,\dots, S_k}$
that
 fit in the diagram
$$
\begin{CD}
\F_n(S_1)\times\cdots\times \F_n(S_k)\times \C_{\set{k}} (M)@>\phi_{S_1,\dots, S_k}>> \C_{(S_1\cup\dots\cup S_k)}(M)\\
@VV\pi V        @VV\pi V \\
M^k @>{\Delta_{S_1}\times\cdots\times \Delta_{S_k}}>> M^{(S_1\cup\dots\cup S_k)}
\end{CD}
$$
and are isomorphisms on $\F_n(S_1)\times\cdots\times \F_n(S_k)\times
\C^0_{\set{k}} (M)$, where $\Delta_S\colon M\to M^S$ are the diagonal
maps. It follows that spaces $\C_{*}(M)$ form a right module over
 the PROP generated by the Fulton-MacPherson operad
$$P(\F_n)(m,l)=\bigcup_{\sum m_i=m} \F_n(m_1)\times\cdots\times \F_n(m_l).$$
This module as a set is freely generated by $\C^0_* (M)$.
The stratification on $\F_n$ defines a stratification on $\C_{*}(M)$.

Denote by $C_*(\C_{\set{k}}(M))$ the complex of rational chains of the Fulton-MacPherson compactification.

\begin{definition}
For a $C_*(F_n)$-algebra $A$ and a compact  parallelized
$n$-manifold without boundary $M$ call the complex
$CM_*(A)=C_*(\C_*(M))\otimes_{C_*(P(F_n))} A$ the {\em manifoldic
chain complex} of $A$ on $M$. Call the homology of the manifoldic
chain complex  the {\em manifoldic homology} of $A$ on $M$.
\label{1}
\end{definition}

This definition is based on Definition 4.14 from \cite{PS}. By
Proposition \ref{eq} one may pass from a $C_*(F_n)$-algebra to an
$e_n$-algebra. As it is shown in \cite{Lu}, the manifoldic homology
is the same as the topological chiral homology with constant
coefficients introduced in {\it loc.~cit} of this $e_n$-algebra.

\subsection{Morphism}

Let $(\g, d)$ be a $\Li$-algebra. Let
$l_{i>1}\colon \Lambda^i \g[i-2] \to \g$ be its higher brackets,
that is, the operations in complex (\ref{complex}) corresponding to the star
trees. The structure of $\Li$-algebra may be encoded in a derivation
$D=D_1+D_2+\dots$ on the free super-commutative algebra generated by
$\g^{\vee}[1]$, where $D_1$ is dual to $d$ and $D_i$ is dual to
$l_i$ on generators and are continued on the whole algebra by the
Leibniz rule. The {\em Chevalley--Eilenberg chain complex} $CE_*(\g)$
is the super-symmetric power $S^*(\g[-1])$ with the differential
$\dtot=d+\theta_2+\theta_3+\dots$, where $\theta_i$ is dual to
$D_i$.

Denote by $[\C^0_{\set{k}}] \in C_*(\C_{\set{k}}(M))$ the chain
given by the submanifold $\C_{\set{k}}^0(M)$ in $\C_{\set{k}}(M)$. 
For a $C_*(F_n)$-algebra $A$ and a cycle $\cc\in
C_*(\C_{\set{k}}(M))$ 
denote by $(a_1\otimes \cdots\otimes a_k)\otimes_{\Sigma_k} \cc\in
CM_*(A)$ the chain given by the tensor product over the
symmetric group. Recall that by (\ref{morphism}) for any
$C_*(F_n)$-algebra $A$ the complex $A[n-1]$ is equipped with a
$\Li$ structure. Denote this $\Li$-algebra by $\LL(A)$. Denote by
$\Alt (a_1\otimes\cdots\otimes a_k)$ the sum 
$\sum_\sigma\pm a_{\sigma(1)}\otimes\cdots\otimes a_{\sigma(k)}$ by all permutations,
where signs are sign given by the sign of the
permutation and the Koszul sign rule.

\begin{prop}
For a $C_*(F_n)$-algebra $A$ and a parallelized compact manifold
without boundary $M$ the map $\TT\colon a_1\wedge\dots\wedge a_k\mapsto
\Alt (a_1\otimes\cdots\otimes a_k)\otimes_{\Sigma_k} [\C^0_{\set{k}}]$ defines a 
morphism from Chevalley--Eilenberg complex
$CE_*(\LL(A))$ to the manifoldic chain complex $CM_*(A)$.
\label{main}
\end{prop}

\begin{proof}
Denote the total differentials on both complexes $CE_*(\LL(A))$ and
$CM_*(A)$ by $\dtot$. One needs to show that $\dtot\circ
\TT=\TT\circ \dtot$.

The border of $ [\C^0_{\set{k}}]$ in $C_*(\C_{\set{k}}(M))$ 
is the sum of all codimension one strata: $\partial [\C^0_{\set{k}}]=\sum_i \theta_i [\C^0_{\set{k-i+1}}] $. 
Here $\theta_i$ is the symmetrization in $C_*(P(\F_n))$
of the operation in $C_*(F_n)$ that corresponds by Proposition \ref{tree} 
to the star with $i$ inputs.
This means that
$$\dtot\circ \TT \,(a_1\wedge\cdots\wedge a_k)=
(d \Alt (a_1\otimes\cdots\otimes
a_k))\mathop{\otimes}\limits_{\Sigma_k}[\C^0_{\set{k}}] +\Alt
(a_1\otimes\cdots\otimes a_k)\mathop{\otimes}_{\Sigma_k}\sum_{i>1}
\theta_i [\C^0_{\set{k-i+1}}]$$ One may carry $\theta$'s from one
factor of $\otimes_{\Sigma_k}$ to another by the very definition of
the tensor product over $C_*(P(F_n))$. And the action of $\theta$'s on
the alternating sum again by definition is  given by the higher brackets of
the $\Li$-algebra. After summing with $d$  it gives the differential on the
Chevalley--Eilenberg complex. It follows that $\dtot\circ
\TT=\TT\circ \dtot$.
\end{proof}

\section{Sketch: invariants of a parallelized manifold and Koszul duality}

\subsection{Invariant of a parallelized manifold}

The idea how to apply  manifoldic homology to manifolds invariant is the following.
Below (Definition \ref{diff}) we sketch a construction of a $e_n$-algebra $\Dn(V)$ such that for any 
$n$-dimensional parallelized compact manifold without boundary $M$
manifoldic homology $HM_*(\Dn(V))$ is one-dimensional
(Proposition \ref{dim}).
Then, 
\begin{equation}
H_*(\LL(\Dn(V)))\to
HM_*(\Dn(V))
\label{class}
\end{equation}
given by Proposition \ref{main} supplies us with a cocycle in the Lie algebra cohomology of $\LL(\Dn(V))$.

In this way we obtain an invariant that is conjecturally related to the universal
Chern--Simons invariant (see  \cite{AS}, \cite{BC}) which takes value in ``graph cohomology'', with the 
Chevalley--Eilenberg   cochain complex
of  $\LL(\Dn(V))$ representing the ``graph complex''.

\subsection{Koszul duality}

Quillen duality (\cite{Q}, \cite{Hi}) 
gives an equivalence between homotopy categories of Lie algebras $Lie$ 
and connected  cocommutative coalgebras $coCom$.
Koszul duality (\cite{Lu},  \cite{Fr}) is an analogous equivalence between the categories of augmented $e_n$-
algebras
and coaugmented $e_n$-coalgebras satisfying certain conditions analogous to connectedness.
I hope to elaborate on these conditions elsewhere. Denote the above mentioned categories by $e_n-alg$
and $e_n-coalg$. The relationship between  Quillen  and Koszul dualities is displayed
in the diagram
\begin{equation}
\xymatrix{
 \LL:\hspace{-30pt}& e_n-alg \ar@{<->}[d]_{\mbox{Koszul duality}} \ar@<-2pt>[r]  & Lie \ar@<-2pt>[l]\ar@{<->}[d]^{\mbox{Quillen duality}} &\hspace{-30pt} : U^n \\
   \Ab:\hspace{-30pt}&e_n-coalg\ar@<-2pt>[r] & coComm \ar@<-2pt>[l] & \hspace{-30pt} :\ii
}
\label{diagram}
\end{equation}
Here, the functor $\LL$ is given by (\ref{morphism}), $U^n$ is the derived universal enveloping $e_n$- algebra 
the functor that is derived left  adjoint to $\LL$, $\ii$ is the embedding of cocommutative coalgebras in $e_n$-
coalgebras and $\Ab$ is its derived right adjoint.

The linear dual of a $e_n$-coalgebra is a $e_n$-algebra. If some $e_n$-algebra and $e_n$-coalgebra
are related by Koszul duality, then the first one and the linear dual of the second one are called
{\em Koszul dual $e_n$-algebras}.

The following statement generalizes the well-known fact that Hochschild homologies of  Koszul dual algebras
are dual to each other (see e.~g. \cite[Appendix D]{Vdb}).
\begin{stat}[Poincar\'e--Koszul duality]
For a 
$n$-dimensional parallelized compact manifold without boundary $M$, the
manifoldic homologies on $M$ of Koszul dual $e_n$-algebras are linear dual to each other. 
\label{pk}
\end{stat}

\subsection{$n$-Weyl algebra}
We say that an element $c$ of a $e_n$-algebra $A$ is  {\em central}, 
if
the product map 
$$
e_n(k+1)\otimes\underbrace{c\otimes A\otimes\cdots\otimes A}_{k+1}\to A 
$$
factors through 
$$
e_n(k+1)\otimes \underbrace{c\otimes A\otimes\cdots\otimes A}_{k+1}\to
 c\otimes e_n(k)\otimes\underbrace{A\otimes\cdots\otimes A}_k .
$$
The latter map is induced by the natural projection
from $k+1$-ary operations to $k$-ary ones. 

By a {\em $e_n$-algebra with curvature} we mean a $e_n$-algebra with a central element
$c$ of degree $n+1$. The condition on $c$ may be relaxed by analogy with 
\cite{LP}. The new condition may be formulated in terms of the deformation complex of an $e_n$-algebra.

Conjecturally, one may define Koszul duality for $e_n$-algebras with curvature in such a way,
that for $n=1$, one recovers Koszul duality for algebras with curvature, see \cite{LP}.

Let $V$ be a graded vector space with a non-degenerate symmetric in the graded sense bilinear form $q$ of 
degree $-(n+1)$. 
Let $S^*(V^\vee)$ be the free graded commutative algebra generated by the vector space dual to $V$.
Denote by $S^*(V^\vee)^\vee$ the restricted dual coalgebra.
By means of inclusion $\ii$ from (\ref{diagram}) consider the pair $(S^*(V^\vee)^\vee, q)$ as a $e_n$-coalgebra 
with curvature.
\begin{definition}
For a graded vector space $V$ with a non-degenerate symmetric in the graded sense 
bilinear form $q$ of degree $-(n+1)$
we denote by $\Dn(V)$ the $e_n$-algebra  Koszul dual to  $(S^*(V^\vee)^\vee, q)$ 
and refer to it as  {\em 
$n$-Weyl algebra} .
\label{diff}	
\end{definition}

\begin{prop} For a 
$n$-dimensional parallelized compact manifold without boundary $M$, the 
manifoldic homology $HM_*(\Dn(V))$ is one-dimensional.
\label{dim}
\end{prop}
\begin{proof} By 
Statement \ref{pk}, $HM_*(\Dn(V))$ is linear dual to 
the manifoldic homology  of the  $e_n$-algebra that is Koszul dual to $\Dn(V)$. Thus, one needs to prove that the 
latter homology is one-dimensional.
The Koszul dual $e_n$-algebra is $e_n$-algebra with curvature $(S^*(V^\vee), q)$.
The manifoldic homology of $S^*(V^\vee)$ is the free commutative algebra generated by $V^\vee\otimes H_*(M)$,
where $H_*(M)$ is homology of $M$ negatively graded. The curvature equips the underlying space of this algebra
with a differential given by multiplication by an element of cohomological degree 1
and of homogeneous degree 2. This element represents the pairing induced by the tensor 
product of $q$ on $V$ and the Poincar\'e paring on $H^*(M)$.
The cohomology of this differential, that is of the de Rham complex of a graded vector space,
is manifoldic homology of the $e_n$-algebra with curvature. As the cohomology of the de Rham complex
is one-dimensional, this implies the proposition.
\end{proof}
\textbf{Example.}
Let $n=1$ and $V$ is concentrated in degree 1. Then $\mathop{\mathfrak{D}^1}(V)$
is the usual Weyl algebra, that is the symplectic Clifford algebra generated by vector
space $V[-1]$ with the skew-symmetric form on it. For $M=S^1$ the manifoldic homology
is the Hochschild homology  and Proposition \ref{dim} matches with
the well-known fact about Weyl algebra:
$$
\dim HH_i(\mathop{\mathfrak{D}^1}(V))=\begin{cases}
1,& i=\dim V,\\
0, & \mbox{otherwise}.
\end{cases} 
$$
Note that this fact is crucially used in \cite{FT} and classes like (\ref{class}) and  (\ref{can}) below restricted to the Lie algebra
of vector fields are exploited there to present the 
Todd class.

\subsection{Concluding remarks} 
Finally, let us look at the morphism from Proposition \ref{main}
from the Koszul duality viewpoint.

For a commutative algebra $C$ there is a canonical morphism
$HH_*(C)\to C$. It may be generalized to manifoldic homology
as follows.

\begin{stat}
For a  homotopy commutative
algebra
(= $e_\infty$-algebra) $C$ and for a 
$n$-dimensional parallelized compact manifold without boundary $M$
there is a canonical map from manifoldic chain complex of $C$ to itself:
\begin{equation}
\pi \colon CM_*(\ii(C))\to C.
\label{pi}
\end{equation}
\end{stat}

Morphism $\pi$  may be constructed by means of manifoldic homology of non-compact manifolds:
every manifold may be embedded $\mathbb R^N$ and as commutative algebra may be considered
as $e_N$-algebra, the embedding induces a morphism of manifoldic homologies, and 
manifoldic homology of $C$ on $\mathbb R^N$ is $C$.

Diagram (\ref{diagram})  shows  that for a Lie algebra $\mathfrak g$ the $e_n$-algebras
$\Ue^n(\mathfrak g)$ and $\ii(CE^*(\mathfrak g))$ are Koszul dual, where $CE^*$ is the  Chevalley--Eilenberg 
cochain
complex. By Poincar\'e--Koszul duality (Statement \ref{pk}) for a 
$n$-dimensional parallelized compact manifold without boundary $M$ 
 homologies $HM_*(\Ue^n(\mathfrak g))$ and $HM_*(\ii(CE^*(\mathfrak g)))$ are dual
to each other.  Morphism (\ref{pi}) gives  $\pi\colon CM_*(\ii(CE^*(\mathfrak g))\to CE^*(\mathfrak g)$
and composing with Poincar\'e--Koszul duality we obtain the map
\begin{equation}
H_*(\mathfrak g)\to HM_*(\Ue^n(\mathfrak g)).
\label{taut}
\end{equation}

Functors $\Ue^n$ and $\LL$ from (\ref{diagram}) being adjoint, there is as canonical morphism
$\Ue^n(\LL(A))\to A$ for any $e_n$-algebra $A$. It induces a map on manifoldic homologies:
\begin{equation}
HM_*(\Ue^n(\LL(A))) \to HM_*(A).
\label{can}
\end{equation}

\begin{stat} The effect of the morphism from Proposition \ref{main} on homologies 
is the composition of (\ref{taut}) for
$\mathfrak {g}=\LL(A)$ and (\ref{can}).
\end{stat}

This morphism may be described even simpler in 
Koszul dual terms. The Koszul dual morphism is the composition
\begin{equation}
CM_*(A^!)\to CM_*(\ii(\Ab(A^!))) \to A^!,
\label{three}
\end{equation}
where the first arrow is induced by the canonical morphism for a pair of adjoint functors $\ii$ and $\Ab$ 
and the second arrow is given by (\ref{pi}).

For our main example $A=\Dn$ the Koszul dual $e_n$-algebra $A^!$ is a $e_n$-algebra
with curvature and the formula (\ref{three}) is not applicable directly. 
It is not clear, what the functor $\Ab$ means for such algebras. The  question is interesting even for
$n=1$, where  $\Ab$ is the derived quotient by the ideal generated by commutators.

%It would be  interesting to find an interpretation  of the effect of $\Ab$ on $e_n$-algebras with curvature
%of the form
%$(C^*(X), c)$, where $C^*(X)$ is the cochain complex of a topological space $X$
%and $c$ is a $(n+1)$-cocycle.
\bibliographystyle{alphanum}
\bibliography{cs}

\end{document}